\documentclass[11pt, amsfonts]{amsart}

\usepackage{amsmath,amssymb,amsthm,color,enumerate,comment,bm}
\usepackage[all]{xy}
\usepackage{tikz}
\usepackage{todonotes}
\usepackage{url} 
\usepackage{hyperref} 

\numberwithin{equation}{section}

\SelectTips{eu}{12}

\newtheorem{thm}{Theorem}[section]

\newtheorem{theorem}[thm]{Theorem}
\newtheorem{proposition}[thm]{Proposition}
\newtheorem{lemma}[thm]{Lemma}
\newtheorem{corollary}[thm]{Corollary}

\theoremstyle{definition}

\newtheorem{definition}[thm]{Definition}

\newtheorem{remark}[thm]{Remark}

\newtheorem{theoremalpha}{Theorem} 

\theoremstyle{remark}

\newcommand{\ZZ}{\mathbb{Z}}

\newcommand{\RR}{\mathbb{R}}

\newcommand{\Ker}{\mathrm{Ker}}

\newcommand{\Cay}{\mathrm{Cay}}

\newcommand{\bC}{\check{\rm C}}

\newcommand{\tv}{\tilde{v}}
\newcommand{\tw}{\tilde{w}}

\newcommand{\N}{\mathcal{N}}

\newcommand{\tB}{\widetilde{B}}

\newcommand{\Bor}{\mathrm{Bor}}
\newcommand{\tBor}{\widetilde{\mathrm{Bor}}}

\newcommand{\tgamma}{\tilde{\gamma}}

\newcommand{\pr}{\mathrm{pr}}

\title[Higher-dimensional generalization of Youngs' theorem]{Higher-dimensional generalization of Youngs' theorem and circular colorings}

\author[K. Enami]{Kengo Enami}
\address{College of Liberal Arts and Science, Kitasato University, 1-15-1 Kitasato, Minami-ku, Sagamihara, Kanagawa 252-0373, Japan
}
\email{enamikengo@gmail.com}

\author[T. Matsushita]{Takahiro Matsushita}
\address{Department of Mathematical Sciences, Faculty of Science, Shinshu University, Matsumoto, Nagano 390-8621, Japan}
\email{matsushita@shinshu-u.ac.jp}



\begin{document}

\baselineskip.525cm

\maketitle

\begin{abstract}
In 1996, Youngs proved that any quadrangulation of the real projective plane is not $3$-chromatic. This result has been extended in various directions over the years, including to other non-orientable closed surfaces, higher-dimensional analogues of quadrangulations
and circular colorings. In this paper, we provide a generalization which yields some of these extensions of Youngs' theorem, by using a method developed in topological combinatorics.
\end{abstract}

\section{Introduction} \label{section introduction}

\subsection{Youngs' theorem} \label{subsection background}
Quadrangulations on surfaces have been extensively studied in topological graph theory. The purpose of this paper is to provide a generalization of a theorem by Youngs (Theorem~\ref{theorem:youngs}) concerning the graph coloring problem of quadrangulations on the real projective plane.

In this paper, every graph is assumed to be simple. A \emph{surface} means a compact connected $2$-dimensional manifold, possibly with boundary. An embedded graph $G$ on a surface $S$ is said to be \emph{cellular} if every face is homeomorphic to a $2$-cell. A cellular embedding is said to be \emph{even-sided} if every face is bounded by a closed walk of even length. In particular, if every face is bounded by a closed walk of length $4$, the embedding is called a \emph{quadrangulation}.

Graph coloring problems are one of the central topics in topological graph theory. For plane graphs, the Four Color Theorem~\cite{AH1,AH2,AH3} serves as a cornerstone result. For graphs embedded on non-spherical closed surfaces, the Map Color Theorem~ \cite{Heawood} is a natural extension of the four color theorem, which states that every graph embedded on a non-spherical closed surface $S$ with Euler characteristic $\varepsilon(S)$ has chromatic number
\[ \chi(G) \leq \left\lfloor\frac{7 + \sqrt{49 - 24\varepsilon(S)}}{2} \right\rfloor.\]
This bound was later shown to be best possible for all closed surfaces except for the Klein bottle, through the work of Ringel and Youngs~\cite{RY}.
These foundational results have inspired a vast body of subsequent research. In addition to these developments, coloring problems for graphs with even-sided embeddings have also received attention.

It is easy to see that a graph cellularly embedded on the sphere is $2$-colorable if and only if the embedding is even-sided.
On the other hand, Hutchinson~\cite{Hutchinson1} showed that every graph $G$ that admits an even-sided embedding on a non-spherical closed surface $S$ has chromatic number
\[ \chi(G)\leq \dfrac{5+\sqrt{25-16\varepsilon(S)}}{2}.\]
Later, Liu et al.~\cite{Liuetal} proved that this bound is best possible for all closed surfaces except for the Klein bottle and the double torus.
In the case of a quadrangulation $G$ on the real projective plane, the above inequality implies that $\chi(G) \le 4$. On the other hand, Youngs proved the following landmark result.

\begin{theorem}[Youngs~\cite{Youngs}]\label{theorem:youngs}
Any quadrangulation of the real projective plane $\RR P^2$ is not $3$-chromatic.
\end{theorem}

Hence the chromatic number of a quadrangulation of the real projective plane is $2$ or $4$.
This phenomenon is remarkable since $3$-chromatic quadrangulations exist for all other non-spherical closed surfaces.
Moreover, Hutchinson~\cite{Hutchinson2} showed that every graph with an even-sided embedding on any orientable closed surface in which every non-contractible cycle is sufficiently long is $3$-colorable.
Youngs’ theorem and its extension to non-orientable closed surfaces mentioned below (Theorem~\ref{theorem:non-ori}) indicate that this does not hold for non-orientable closed surfaces.

\subsection{Various generalizations and refinements} \label{subsection generalizations}
There are various extensions of Youngs' theorem, and our main theorem (Theorem~\ref{theorem main}) is a generalization of Youngs' theorem, which yields some of them.

First, we discuss the extension of Youngs' theorem to non-orientable closed surfaces, which was independently established by Archdeacon et al.~\cite{AHNNO} and Mohar and Seymour~\cite{MS}.

\begin{theorem}[Archdeacon et al.~\cite{AHNNO} and Mohar and Seymour~\cite{MS}] \label{theorem:non-ori}
Let $N$ be a closed non-orientable connected surface, and $G$ a quadrangulation of $N$.
Suppose that $G$ has an odd cycle $C$ such that the surface obtained by cutting $N$ along $C$ is orientable.
Then $G$ is not $3$-chromatic.
\end{theorem}

Next we discuss a higher-dimensional generalization of Youngs' theorem.
Hachimori, Nakamoto and Ozeki~\cite{HNO} considered $d$-dimensional quadrangulations of the real $d$-dimensional projective space $\RR P^d$ for any $d\geq 2$ as follows:
Set $I = [0,1]$ and consider the standard CW structure of $I^n$.
A regular CW complex $C$ is called a \emph{cubical complex} if each characteristic map $\varphi \colon I^n \to C$ of an $n$-cell of $C$ is a cellular embedding, that is, $\varphi(I^n)$ is a subcomplex of $C$ and $\varphi$ is an isomorphism between $I^n$ and $\varphi(I^n)$.
If a cubical complex $C$ is homeomorphic to a topological space $X$, then the 1-skeleton of $C$ is called a \emph{quadrangulation} of $X$.

Hachimori, Nakamoto and Ozeki~\cite{HNO} proved that for all $d \geq 2$, any $d$-dimensional quadrangulation of the real projective space $\RR P^d$ satisfying a certain geometric condition is $2$- or $4$-chromatic. They further asked whether this geometric condition is necessary; that is, whether every non-bipartite $d$-dimensional quadrangulation of $\RR P^d$ is necessarily $4$-chromatic. Recently, Kaiser et al.~\cite{Kaiseretal} proved a negative answer to this question by constructing $3$-dimensional quadrangulations of $\RR P^3$ with arbitrarily large chromatic numbers. Nevertheless, they also established the following theorem, which shows that Youngs’ result does extend to higher dimensions under certain conditions.

\begin{theorem}[{Kaiser et al. \cite{Kaiseretal}}]\label{theorem:high-dim}
For all $d\geq 2$, every $d$-dimensional quadrangulation of the real projective space $\RR P^d$ is not $3$-chromatic.
\end{theorem}

Another type of higher-dimensional generalization of Youngs' theorem was discussed in Kaiser and Stehl\'{i}k~\cite{KS}.

Finally, we discuss a refinement of Youngs' theorem from the perspective of circular chromatic numbers. The circular chromatic number $\chi_c(G)$ is a real-valued invariant of a graph $G$ satisfying $\lceil \chi_c(G) \rceil = \chi(G)$. In this sense, $\chi_c(G)$ refines the chromatic number, and has been extensively studied. We refer the reader to \cite{PW, Zhu, Zhu2} for surveys on this concept.
The precise definition of the circular chromatic number will be given in Subsection~\ref{subsection Borsuk graph}.
DeVos et al.~\cite{DGMVZ} proved the following refinement of Youngs' theorem.

\begin{theorem}[DeVos et al.~\cite{DGMVZ}]\label{theorem:circular}
Let $G$ be a non-bipartite graph with an even-sided embedding on $\RR P^2$.
Let $k$ be a positive integer such that every face of $G$ is a $2i$-gon with $i \le k$.
Then
\[ \chi_c(G) \ge 2 + \frac{2}{k-1}.\]
\end{theorem}

\subsection{Results}

The goal of this paper is to present a generalization of Youngs' theorem that yields Theorems~\ref{theorem:non-ori}, \ref{theorem:high-dim} and \ref{theorem:circular}. Before stating the main theorem, we define a class of CW complexes. Let $X$ be a CW complex satisfying the following conditions:

\begin{enumerate}[(1)]
\item The $1$-skeleton $X^1$ of $X$ is a simple graph.

\item The attaching map of each $2$-cell of $X$ is a graph homomorphism $C_{2r} \to X^1$ for some $r \ge 2$.
\end{enumerate}
In this case, we say that $X$ is a \emph{CW complex with even $2$-skeleton}. When the attaching map of every $2$-cell is a graph homomorphism from $C_4$, we say that $X$ is a \emph{CW complex with quadrangulated $2$-skeleton}. Note that a cubical complex mentioned in Subsection~\ref{subsection generalizations} is a CW complex with quadrangulated $2$-skeleton.

Now we are ready to state our main theorem of this paper, which is a generalization of Theorems~\ref{theorem:non-ori}, \ref{theorem:high-dim} and \ref{theorem:circular}. To see that the following theorem is a generalization of Theorem~\ref{theorem:non-ori}, see Proposition~\ref{proposition cut}.

\begin{theoremalpha} \label{theorem main}
Let \( X \) be a CW complex with even 2-skeleton, $k$ an integer at least $2$, and suppose that each attaching map \( C_{2i} \to X^1 \) satisfies \( i \le k \).
Assume that there is a closed walk \( \gamma \) in \( X^1 \) with odd length such that the integral homology class $[\gamma]_{\ZZ}$ represented by $\gamma$ is a torsion element in $H_1(X ; \ZZ)$. Then,
\[ \chi_c(X^1) \ge 2 + \frac{2}{k - 1}. \]
\end{theoremalpha}

Recall that $X^1$ contains an odd closed walk if and only if it is non-bipartite. Note that if $H_1(X; \ZZ)$ is a torsion group, then the condition that $[\gamma]_{\ZZ}$ is a torsion element in $H_1(X; \ZZ)$ is automatically satisfied.

Restricting our attention to quadrangulations and chromatic numbers, we obtain the following theorem from Theorem~\ref{theorem main}:

\begin{theoremalpha} \label{theorem main 2}
Let $X$ be a CW complex with quadrangulated $2$-skeleton. Suppose that there is a closed walk $\gamma$ of $X^1$ with odd length such that the integral homology class $[\gamma]_{\ZZ}$ represented by $\gamma$ is a torsion element in $H_1(X ; \ZZ)$. Then $X^1$ is not $3$-chromatic.
\end{theoremalpha}

Theorem~\ref{theorem main 2} is a generalization of Theorems~\ref{theorem:non-ori} and \ref{theorem:high-dim}. In Section~3, we provide a short proof of Theorem~\ref{theorem main 2} using classical algebraic topology. As a by\mbox{-}product of this proof, we show that for any CW complex with even 2-skeleton such that $H_1(X; \ZZ_2)$ is trivial, the 1-skeleton $X^1$ of $X$ is bipartite (Corollary~\ref{corollary bipartite}). This is a further generalization of the fact that a cellular even-sided embedding on the $2$-sphere is bipartite.

In the proof of Theorem~\ref{theorem main}, we use the discretized fundamental group introduced by the second author \cite{MatsushitaJMSUT} in the study of the fundamental group of Lov\'asz's neighborhood complex \cite{Lovasz}. For a positive integer $k$ and a based graph $(G,v)$, we associate $\pi_1^k(G,v)$, called the $k$-fundamental group. The preliminaries of $k$-fundamental groups will be given in Section~4. As a \mbox{by-product} of our methods, we can refine some results from the recent work of Krebs and Sankar \cite{KrebsSankar} concerning the graph coloring problem of Cayley graphs in terms of circular chromatic numbers (see Corollary~\ref{corollary KS} and Theorem~\ref{theorem refinement}).

\subsection*{Organization of the paper}
In Section~\ref{section preliminaries}, we provide preliminaries of graphs and algebraic topology. Section~\ref{section chromatic} is devoted to the proof of Theorem~\ref{theorem main 2}. In Section~4, we provide preliminaries of $k$-fundamental groups, following \cite{MatsushitaJMSUT} for the proof of Theorem~\ref{theorem main}. In Section~5, we recall the definition of circular chromatic number and provide the proof of Theorem~\ref{theorem main}. In Section~5, we discuss the relationship between our results and the recent work by Krebs and Sankar \cite{KrebsSankar}.

\section{Preliminaries} \label{section preliminaries}

In this section, we review necessary definitions and facts, following \cite{GR, Hatcher, HN, Kozlov}.

\subsection{Graphs}

Throughout this paper, all graphs are assumed to be simple; that is, a graph is a pair $(V, E)$, where $V$ is a set and $E$ is a family of $2$-element subsets of $V$. Note that graphs are not assumed to be finite.

For graphs $G$ and $H$, a \emph{graph homomorphism from $G$ to $H$} is a map $f \colon V(G) \to V(H)$ such that $\{ v,w\} \in E(G)$ implies $\{ f(v), f(w)\} \in E(H)$. Let $K_n$ denote the complete graph with $n$ vertices, i.e., $V(K_n) = \{ 1, \cdots, n\}$ and $i,j \in V(K_n)$ are adjacent if and only if $i \ne j$. An \emph{$n$-coloring of a graph $G$} is a graph homomorphism from $G$ to $K_n$. For $n \ge 3$, the \emph{$n$-cycle graph $C_n$} is defined by $V(C_n) = \ZZ / n$ and $E(C_n) = \{ \{ x,x +1 \} \mid x \in \ZZ / n\}$.

\subsection{Basic facts from algebraic topology}

For basic definitions and facts of CW complexes, we refer the reader to Chapter~0 of \cite{Hatcher}. For a CW complex $X$ and for a non-negative integer $k$, the \emph{$k$-skeleton of $X$} is the subcomplex of $X$ consisting of all $i$-cells of $X$ such that $i \le k$, and is denoted by $X^k$.

\begin{proposition}[{see \cite[Example~1B.1 and Proposition 4.1]{Hatcher}}] \label{proposition KG1}
For every connected graph $G$ and for every integer $i$ at least $2$, $\pi_i(G)$ is trivial; that is, for every $i \ge 2$, every continuous map from $S^i$ to $G$ is null-homotopic.
\end{proposition}

For a $k$-cell $\sigma$ of a CW complex $X$, let $\Phi_\sigma \colon S^{k-1} \to X^{k-1}$ denote the attaching map of $\sigma$. Note that a continuous map $f \colon S^{k-1} \to Y$ is null-homotopic if and only if there is a continuous map $f' \colon D^k \to Y$ such that $f'|_{S^{k-1}} = f$. By the definition of CW complexes, we have the following proposition.

\begin{proposition}[{See also \cite[Lemma~4.7 and its proof]{Hatcher}}] \label{proposition extension}
Let $X$ be a CW complex, $Y$ a path-connected topological space, and $f \colon X^k \to Y$ a continuous map. Suppose that for each $(k+1)$-cell $\sigma$ the composition $f \circ \Phi_\sigma \colon S^k \to Y$ is null-homotopic. Then there is a continuous map $f' \colon X^{k+1} \to Y$ such that $f' |_{X^k} = f$.
\end{proposition}

Here we review the cellular homology. See \cite[Section~2.2, Section~3.A]{Hatcher} for details. For a CW complex $X$ and for a commutative ring $R$ with unit, let $C_\bullet(X ; R)$ be the cellular chain complex of $X$. For a non-negative integer $k$, $C_k(X; R)$ can be identified with the free $R$-module generated by the set of $k$-cells of $X$. We set $C_k(X; R) = 0$ for $k < 0$. We can define an $R$-linear map $\partial_k \colon C_k(X; R) \to C_{k-1}(X ; R)$, called the \emph{boundary map}, which satisfies $\partial_{k-1} \circ \partial_k = 0$, and the homology group of $C_k(X; R)$ can be naturally identified with the (singular) homology group $H_k(X; R)$ of $X$.

Here we explain in detail the boundary map in the following case: $X^1$ is a simple graph, $R = \ZZ_2$, and the attaching map of a $2$-cell $\sigma$ is a graph homomorphism $\Phi_\sigma \colon C_n \to X^1$. Then $\sigma \in C_2(X; \ZZ_2)$ and $\partial \sigma \in C_1(X; \ZZ_2)$. Hence $\partial \sigma$ is written as
\[ \partial \sigma = a_1 e_1 + \cdots + a_m e_m,\]
where $a_i \in \ZZ_2$ and $e_i \in E(X^1)$. The coefficient $a_i$ is determined by
\[ a_i = \# \{ e \in E(C_{n}) \mid \Phi_\sigma(e) = e_i \} \quad \mod 2.\]
See \cite[Proposition~2.3, Cellular boundary formula in Section~2.2]{Hatcher}.

In this paper, we regard a circle $S^1$ as $\RR / \ZZ$, and $0$ as its basepoint. Let $\gamma \colon (S^1, 0) \to (X,x_0)$ be a loop of a based space $(X,x_0)$. Let $[S^1] \in H_1(S^1; \ZZ)$ be the fundamental class of $S^1$, and $\gamma_*([S^1]) \in H_1(X; \ZZ)$. This correspondence yields the following group homomorphism
\[ h \colon \pi_1(X, x) \to H_1(X; \ZZ), \quad [\gamma] \mapsto \gamma_*([S^1]),\]
which is called the Hurewicz map.

\begin{theorem}[{\cite[Theorem~2A.1]{Hatcher}}] \label{theorem abelianization}
Let $X$ be a path-connected space and $x_0$ a point of $X$. Then the Hurewicz map induces a group isomorphism
\[ \pi_1(X, x_0) / [\pi_1(X, x_0), \pi_1(X, x_0)] \xrightarrow{\cong} H_1(X; \ZZ).\]
\end{theorem}

By the cellular approximation theorem, we have the following theorem:

\begin{proposition}[{see \cite[Corollary~4.12]{Hatcher}}] \label{proposition 2-connected}
Let $X$ be a CW complex, $A$ a subcomplex of $X$, $x_0$ a vertex of $A$, and $\iota$ the inclusion map from $A$ to $X$. Then the following hold:
\begin{enumerate}[(1)]
\item If $X^2 = A^2$, then $\iota_* \colon \pi_1(A,x_0) \to \pi_1(X,x_0)$ is an isomorphism.

\item If $X^1 = A^1$, then $\iota_* \colon \pi_1(A,x_0) \to \pi_1(X,x_0)$ is surjective.
\end{enumerate}
\end{proposition}

\subsection{Homology classes of simple closed curves in surfaces}

In relation to Theorem~\ref{theorem:non-ori}, here we discuss the homology classes of some simple closed curves in non-orientable closed surfaces.

Let $M$ be a surface, possibly having a boundary. A \emph{simple closed curve in $M$} is an injective continuous map $\gamma \colon S^1 \to M$. Let $[S^1]_R \in H_1(S^1; R)$ be the fundamental class of $S^1$. We call $\gamma_*([S^1]_R) \in H_1(M; R)$ the \emph{homology class of $M$ determined by $\gamma$}. We write $[\gamma]_R$ to indicate this homology class $\gamma_*([S^1]_R)$.

The following proposition is a well-known fact in topology, and is deduced from \cite[Exercise~31 of Section~3.3]{Hatcher} and the homology exact sequence for pairs.

\begin{proposition} \label{proposition orientable}
Let $M$ be a compact orientable surface whose boundary $\partial M$ is homeomorphic to $S^1$. Then the homology class determined by $\partial M$ in $H_1(M; \ZZ)$ is $0$.
\end{proposition}

This proposition yields the following:

\begin{proposition} \label{proposition cut}
Let $N$ be a non-orientable compact surface, $\gamma$ a simple closed curve of $N$ such that the surface obtained by cutting $N$ along $\gamma$ is orientable. Then, we have $2[\gamma]_{\ZZ} = 0$.
\end{proposition}
\begin{proof}
Let $M$ be the surface obtained by cutting $N$ along $\gamma$. Then the boundary of $M$ is homeomorphic to $S^1$ and there is a natural continuous map $f \colon M \to N$ such that the restriction to $\partial M$ is the double cover $S^1 \to S^1$. Hence we have $0 = f_* ([\partial M]_{\ZZ}) = \pm 2  [\gamma]_{\ZZ}$ by Proposition~\ref{proposition orientable}.
\end{proof}

\section{Chromatic numbers} \label{section chromatic}

In this section, we provide a short proof of Theorem~\ref{theorem main 2}. Let $G$ be a graph, $k$ an integer at least $3$ and $f \colon G \to K_k$ a graph homomorphism. A \emph{rainbow square} is a $4$-cycle $C_4$ in $G$ such that $f(i) \ne f(j)$ for every distinct pair of vertices $i, j$ of $C_4$. The goal of this section is to prove the following theorem, which is a generalization of Theorem~\ref{theorem main 2}.

\begin{theorem} \label{theorem chromatic}
Let $X$ be a CW complex with quadrangulated $2$-skeleton, $k$ an integer at least $3$ and $f \colon X^1 \to K_k$ a $k$-coloring of $X^1$. Suppose that there is a closed walk $\gamma$ of $X^1$ with odd length such that the integral homology class $[\gamma]_\ZZ$ represented by $\gamma$ is a torsion. Then $X^1$ has a rainbow square. In particular, $\chi(X^1) \ge 4$.
\end{theorem}

We begin the proof with the following proposition.

\begin{proposition} \label{proposition odd cycle}
Let $X$ be a CW complex with even $2$-skeleton. Let $\gamma$ be a closed walk of $X^1$ with odd length. Then the $\ZZ_2$-homology class $[\gamma]_{\ZZ_2}$ represented by $\gamma$ is non-zero.
\end{proposition}
\begin{proof}
Let $C_\bullet(X; \ZZ_2)$ be the cellular chain complex with $\ZZ_2$-coefficient of $X$. Define the homomorphism $\Phi \colon C_1(X; \ZZ_2) \to \ZZ_2$ by $\Phi([\sigma]) = 1$ for every $1$-cell $\sigma$ of $X$. Since $X$ is a CW complex with even $2$-skeleton, the composition
\[ C_2(X ; \ZZ_2) \xrightarrow{\partial} C_1(X; \ZZ_2) \xrightarrow{\Phi} \ZZ_2 \]
is zero. Since $\gamma$ is a closed walk with odd length, we have $\Phi(\gamma) \ne 0$, and hence $\gamma$ is not contained in $\partial C_2(X; \ZZ_2)$. This means that the $\ZZ_2$-homology class $[\gamma]_{\ZZ_2}$ is non-zero.
\end{proof}

This proposition yields the following noteworthy corollary, which generalizes the well-known fact that every even-sided cellular embedding on \( S^2 \) is bipartite.

\begin{corollary} \label{corollary bipartite}
Let $X$ be a CW complex with even $2$-skeleton. If $H_1(X; \ZZ_2) = 0$, then the $1$-skeleton $X^1$ of $X$ is bipartite.
\end{corollary}

Here we note the following lemma, which is easily verified:

\begin{lemma} \label{lemma null square}
Let $k$ be an integer and $f \colon C_4 \to K_k$ a graph homomorphism. If there are distinct $i$, $j \in V(C_4)$ such that $f(i) = f(j)$, then $f$ is null-homotopic.
\end{lemma}

We are now ready to prove Theorem~\ref{theorem chromatic}.

\begin{proof}[Proof of Theorem~\ref{theorem chromatic}]
Let  $f \colon X^1 \to K_k$ be a graph homomorphism and suppose that $f$ has no rainbow square. Since $X$ is a CW complex with quadrangulated $2$-skeleton, an attaching map of a $2$-cell $\sigma$ of $X$ is a graph homomorphism $\varphi_\sigma \colon C_4 \to X^1$. Then Lemma~\ref{lemma null square} implies that $\varphi_\sigma$ is null-homotopic. This means that there is a continuous map $f' \colon X^2 \to K_k$ such that $f'|_{X^1} = f$ (Proposition~\ref{proposition extension}). Since $\pi_i(K_k) = 0$ for every $i \ge 2$ (Proposition~\ref{proposition KG1}), it follows that there is a continuous map $F \colon X \to K_k$ such that $F|_{X^2} = f'$.

We now consider $F_*([\gamma]_{\ZZ})$ in $H_1(K_k ; \ZZ)$. Since $\tau = [\gamma]_{\ZZ}$ is a torsion element and $H_1(K_k; \ZZ) \cong \ZZ^{\frac{(k-1)(k-2)}{2}}$, we have $F_*([\gamma]_{\ZZ}) = 0$. On the other hand, we can show that $F_*([\gamma]_{\ZZ})$ is non-zero as follows: First, note that $F_*([\gamma]_{\ZZ}) = [F \circ \gamma]_{\ZZ} = [f \circ \gamma]_{\ZZ}$. Hence $F_*([\gamma]_{\ZZ})$ is the homology class represented by the closed walk $f \circ \gamma$ with odd length in $K_k$. Since the natural homomorphism $H_1(X; \ZZ) \to H_1(X; \ZZ_2)$ sends $[f \circ \gamma]_{\ZZ}$ to $[f \circ \gamma]_{\ZZ_2}$, Proposition~\ref{proposition odd cycle} implies that $F_*([\gamma]_\ZZ) = [f \circ \gamma]_{\ZZ}$ is non-zero. This is a contradiction.
\end{proof}

\begin{remark} \label{remark archdeaconetal}
Archdeacon et al.~\cite{AHNNO} proved Theorem~\ref{theorem:non-ori} by showing the following stronger theorem: Let $G$ be a quadrangulation of a non-orientable closed surface $N$ with an odd cycle $C$ that cuts open $N$ into an orientable surface.
Then, for any coloring of $G$, there exists at least one face whose four corners get four distinct colors (see \cite[Theorem 9]{AHNNO}). Theorem~\ref{theorem chromatic} is also a generalization of this result.
\end{remark}

\begin{remark} \label{remark kaiser}
Kaiser et al.~\cite{Kaiseretal} in fact showed the following generalizing form of Theorem~\ref{theorem:high-dim}. Suppose that $X$ is a $d$-dimensional closed manifold satisfying the following conditions. Here, $\smile$ denotes the cup product.
\begin{enumerate}[(1)]
\item For any non-trivial element $\alpha$ of $H^{d-2}(X; \ZZ_2)$, there exists $\beta \in H^2(X; \ZZ_2)$ such that $\alpha \smile \beta \ne 0$.

\item For any non-trivial elements $\alpha, \beta \in H^1(X; \ZZ_2)$, $\alpha \smile \beta \ne 0$.
\end{enumerate}
Then, \cite[Theorem~4.1]{Kaiseretal} states that any quadrangulation of $X$ is not 3-chromatic.

Using Proposition~\ref{proposition odd cycle}, we can strengthen
their result \cite[Theorem~4.1]{Kaiseretal} as follows: If $X$ is a CW complex with quadrangulated 2-skeleton and its $1$-skeleton $X^1$ is $3$-chromatic, then there is a non-trivial cohomology class $\alpha \in H^1(X;\ZZ_2)$ such that $\alpha \smile \alpha = 0$. To see this, suppose that $X^1$ is $3$-chromatic. Then, there is a graph homomorphism $f \colon X^1 \to K_3$. Since $X^1$ is non-bipartite, $X^1$ has an odd cycle $\gamma$. Then $f_* \colon H_1(X^1; \ZZ_2) \to H_1(K_3; \ZZ_2)$ sends $[\gamma]_{\ZZ_2}$ to $[f \circ \gamma]_{\ZZ_2}$, and by Proposition~\ref{proposition odd cycle}, these homology classes are non-trivial. By the universal coefficient theorem (see \cite[Section~3.1]{Hatcher}), the map $f^* \colon H^1(K_3; \ZZ_2) \to H^1(X; \ZZ_2)$ is non-zero. Let $\beta \in H^1(K_3; \ZZ_2)$ be the generator. Then $\alpha = f^*(\beta)$ is non-trivial, and $\displaystyle \mbox{$\alpha \smile \alpha = f^*(\beta) \smile f^*(\beta)$} = \mbox{$f^*(\beta \smile \beta$)} = 0$. This completes the proof.

Let $m$ be a positive integer, $l_1, \cdots, l_n$ integers relatively prime to $m$, and assume that $n \ge 2$. Then the lens space $L_m(l_1, \cdots, l_n)$ is defined as an orbit space of a free action of $\ZZ_m$ on $S^{2n-1}$(see \cite[Example~2.43]{Hatcher}). If $m$ is divisible by $4$, then the square $\alpha \smile \alpha$ of the generator $\alpha$ of $H^1(L_m(l_1, \cdots, l_n); \ZZ_2) \cong \ZZ_2$ vanishes.

Indeed, suppose that $m$ is even. Considering the generators of cellular cohomology of standard cell structure (see \cite[Example~2.43]{Hatcher}), the map
\[\Phi \colon \ZZ_m \cong H^i(L_m(l_1, \cdots, l_n); \ZZ_m) \to H^i(L_m(l_1, \cdots, l_n); \ZZ_2) \cong \ZZ_2\]
induced by the ring homomorphism $\ZZ_m \to \ZZ_2$ is the surjective homomorphism for $i = 0,1, \cdots, 2n - 1$. Let $\alpha'$ be a generator of $H^1(L_m(l_1, \cdots, l_n); \ZZ_m)$ and $\beta'$ a generator of $H^2(L_m(l_1, \cdots, l_n); \ZZ_m)$. Then by \cite[Example~2.43]{Hatcher}, we have $\alpha' \smile \alpha' = 2^{-1}m \beta'$.

Hence if $m$ is even and not divisible by $4$, we have that $\alpha \smile \alpha = \Phi(\beta')$ is the generator of $H^2(L_m(l_1, \cdots, l_n); \ZZ_2)$, and hence we have
\[ H^*(L_m(l_1, \cdots, l_n); \ZZ_2) \cong H^*(\RR P^{2n-1} ; \ZZ_2).\]
Thus, in this case $L_m(l_1, \cdots, l_n)$ satisfies conditions (1) and (2) above, and \cite[Theorem~4.1]{Kaiseretal} can be applied to it.

On the other hand, suppose that $m$ is divisible by $4$. Then,
\[ \alpha \smile \alpha = \Phi(\alpha' \smile \alpha') = 2^{-1}m \cdot \Phi(\beta) = 0.\]
in $H^2(L_m(l_1, \cdots, l_n); \ZZ_2)$.
Hence \cite[Theorem~4.1]{Kaiseretal} cannot be applied to this case\noindent
\footnote{There is a minor error in \cite{Kaiseretal}. In the preceding paragraph of \cite[Corollary~4.3]{Kaiseretal}, they stated that the cohomology ring of $H^*(L_m(l_1, \cdots, l_n); \ZZ_2)$ is isomorphic to $H^*(\RR P^{2n-1}; \ZZ_2)$ when $m$ is even. However, this is true only when $m$ is even and not divisible by $4
$, and hence \cite[Theorem~4.1]{Kaiseretal} cannot be applied to the case that $m$ is divisible by $4$.}.
However, by Theorem~\ref{theorem chromatic}, the $1$-skeleton of any CW structure of $L_m(l_1, \cdots, l_n)$ with quadrangulated $2$-skeleton is not $3$-chromatic. We also note that when $m$ is odd, Corollary~\ref{corollary bipartite} implies that the $1$-skeleton of any CW structure $X$ of $L_m(l_1, \cdots, l_n)$ with quadrangulated $2$-skeleton is bipartite.
\end{remark}

Hence, by Remark~\ref{remark kaiser}, the following corollary to Theorem~\ref{theorem main 2} is a result strictly stronger than the previous result by \cite[Corollary~4.3]{Kaiseretal}.

\begin{corollary}
Let $m$ be a positive integer, $n$ an integer at least $2$, and $l_1, \cdots, l_n$ integers relative to $m$. Then, every quadrangulation of $L_m(l_1, \cdots, l_n)$ is not $3$-chromatic.
\end{corollary}

Note that, using Theorem~\ref{theorem main}, we can obtain a more refined version concerning the circular chromatic number.

\section{$k$-fundamental groups}

In this section, we provide preliminaries on $k$-fundamental groups, following \cite{MatsushitaJMSUT}. The $k$-fundamental group was introduced in \cite{MatsushitaJMSUT} in the study of the fundamental group of Lov\'asz's neighborhood complex \cite{Lovasz}, and was applied to multiplicative graphs in \cite{TWSIAM, WrochnaJCTB} and the homomorphism reconfiguration problem \cite{Matsushita_arxiv}. Throughout this section, we assume that $k$ is a fixed positive integer.

\subsection{$k$-fundamental groups} \label{subsection k-fundamental group}
Let $P_n$ be the path graph with $n + 1$ vertices. Namely, $V(P_n) = \{ 0,1,\ldots, n\}$ and $E(P_n) = \{ (x,y) \mid |x - y| = 1\}$. A \emph{walk in a graph $G$ of length $n$} is a graph homomorphism $\gamma \colon P_n \to G$, and we set $\mathrm{length}(\gamma) = n$. For a pair of walks $\gamma \colon P_n \to G$ and $\delta \colon P_m \to G$ such that $\gamma(n) = \delta(0)$, the \emph{concatenation $\gamma \cdot \delta \colon P_{m+n} \to G$ of $\gamma$ and $\delta$} is defined by
\[ \gamma \cdot \delta (i) = \begin{cases}
\gamma(i) & (0 \le i \le n)\\
\delta(i - n) & (n \le i \le n + m).
\end{cases}\]
For a walk $\gamma \colon P_n \to G$ of $G$, we define the walk $\bar{\gamma} \colon P_n \to G$ by $\bar{\gamma}(i) = \gamma(n-i)$.

Let $v,v' \in V(G)$, and let $\Omega (G,v,v')$ be the set of walks joining $v$ to $v'$. Let $\simeq_k$ be the equivalence relation on $\Omega(G,v,v')$ generated by the following two relations (A) and (B$)_k$: For walks $(x_0, \cdots, x_m)$ and $(y_0, \cdots, y_n)$ joining $v$ to $v'$,
\begin{enumerate}[(A)]
\item[(A)] $n = m + 2$ and there is $j \in \{ 0,1,\ldots, m\}$ such that $x_i = y_i$ for every $i \le j$ and $x_i = y_{i+2}$ for every $i \ge j$.

\item[(B)$_k$] $m = n$ and $\# \{ i \mid x_i \ne y_i\} < k$.
\end{enumerate}

Let $\pi_1^k(G,v,v')$ be the quotient set $\Omega(G,v,v') / {\simeq_k}$. We call the equivalence class containing $\gamma \in \Omega(G,v,v')$ the \emph{$k$-homotopy class of $\gamma$}, and is denoted by $[\gamma]_k$. We write $\pi_1^k(G,v)$ instead of $\pi_1^k(G,v,v)$. The concatenation of walks defines a group operation on $\pi_1^k(G,v)$. We call this group the \emph{$k$-fundamental group of the based graph $(G,v)$}.

By the definition of the $k$-fundamental group $\pi_1^k(G,v)$, the following is a well-defined group homomorphism
\begin{equation} \label{parity hom}
\pi_1^k(G,v) \to \ZZ / 2, \quad [\gamma] \mapsto \mathrm{length}(\gamma)\mod 2.
\end{equation}

\begin{remark} \label{remark parity}
In particular, for every closed walk $\gamma$ with odd length, the $k$-homotopy class $[\gamma]_k$ of $\gamma$ is non-trivial in $\pi_1^k(G,v)$.
\end{remark}

The kernel of the group homomorphism described in \eqref{parity hom} is called the \emph{even part of $\pi_1^k(G,v)$} and is denoted by $\pi_1^k(G,v)_{ev}$. If the component of $G$ containing $v$ is bipartite, then $\pi_1^k(G,v)_{ev}$ coincides with $\pi_1^k(G,v)$. If the component of $G$ containing $v$ is non-bipartite, then $\pi_1^k(G,v)_{ev}$ is a subgroup of $\pi_1^k(G,v)$ whose index is $2$.

Note that a basepoint-preserving graph homomorphism $f \colon (G,v) \to (H,w)$ induces a group homomorphism $f_* \colon \pi_1^k(G,v) \to \pi_1^k(H,f(v))$, which preserves the parities.

\begin{theorem}[{\cite[Theorem~3.2]{MatsushitaJMSUT}}] \label{theorem X_k}
Let $(G,v)$ be a based graph. Let $X_k(G)$ be the $2$-dimensional CW complex defined as follows:
\begin{enumerate}[(1)]
\item The $1$-skeleton of $X_k(G)$ is $G$.

\item For every graph homomorphism $f \colon C_{2i} \to G$ with $i \le k$, we attach a $2$-cell to $G$ along $f$. Conversely, every $2$-cell of $X_k(G)$ is obtained in this way.
\end{enumerate}
Then there is a natural isomorphism $\Psi \colon \pi_1^k(G,v) \xrightarrow{\cong} \pi_1(X_k(G), v)$.
\end{theorem}

This isomorphism $\Psi$ is described as follows: Let $\gamma \in \Omega(G,v)$ be a closed walk of $(G,v)$. Then $\Psi([\gamma]_k)$ is the homotopy class of the closed curve $\hat{\gamma} \colon [0,1] \to X_k(G)$ defined by
\[ \hat{\gamma} \left( \frac{i + t}{n} \right) = (1-t)\gamma(i) + t \gamma(i+1),\]
where $i$ is an integer such that $0 \le i < n$ and $t$ is a real number with $0 \le t \le 1$.
Here we consider the right-hand side of the above equation as the point of the edge $\{ \gamma(i), \gamma(i+1)\}$ divided into $1-t$ and $t$.

Here, we discuss the relationship between $k$-fundamental groups and Lov\'asz's neighborhood complex $\N(G)$ of \cite{Lovasz}.

\begin{definition}[{see \cite[Section~4]{MatsushitaJMSUT}}] \label{definition k-neighborhood}
Let $G$ be a graph, $x$ a vertex of $G$, and $k$ a positive integer. Define $N_G(x)$ by $N_G(x) = \{ y \in V(G) \mid (x,y) \in E(G) \}$. For a positive integer $k$, we define the \emph{$k$-neighborhood $N_G^k(x)$ of $x$} recursively by
\[ N^1_G(x) = N_G(x), \quad N^{k+1}_G(x) = \bigcup_{y \in N^k_G(x)} N_G(y).\]
We write $N(x)$ (or $N^k(x)$) instead of $N_G(x)$ (or $N^k_G(x)$, respectively) if there is no risk of confusion.
\end{definition}

For $j \ge 1$, the \emph{$j$-neighborhood complex $\N^j(G)$ of a graph $G$} is defined to be the (abstract) simplicial complex whose vertex set is $V(G)$ and whose simplices are finite subsets of $V(G)$ contained in $N^j(x)$ for some $x \in V(G)$. Then, the neighborhood complex $\N(G)$ coincides with the $1$-neighborhood complex $\N^1(G)$.

\begin{theorem}[{\cite[Theorem~1.1]{MatsushitaJMSUT}}] \label{theorem neighborhood complex}
Let $(G,v)$ be a graph with a non-isolated basepoint $v$, and let $j$ be a positive integer. Then there is a natural isomorphism
\[
\pi^{2j}_1(G,v)_{ev} \xrightarrow{\cong} \pi_1(\N^j(G),v).
\]
\end{theorem}

\subsection{$k$-covering maps} \label{subsection k-covering}

There is a close relationship between the ordinary fundamental group and covering spaces: it is well known that to each subgroup of the fundamental group, there corresponds a connected based covering space. The covering space corresponding to the $k$-fundamental group is the $k$-covering map described below (Definition~\ref{definition k-covering}).

Note that $w \in N^k(v)$ if and only if there is a walk $\gamma$ joining $v$ to $w$ of length $k$.

\begin{definition}[{see \cite[Definition~6.1]{MatsushitaJMSUT}}] \label{definition k-covering}
A graph homomorphism $p \colon G \to H$ is called a \emph{$k$-covering map} if for every $x \in V(G)$ and for every integer $i$ with $1 \le i \le k$, the map $p|_{N^i_G(x)} \colon N^i_G(x) \to N^i_H(p(x))$ is bijective.
\end{definition}

A covering map of a graph in the classical sense is a $1$-covering map (see \cite{GR}). The following lemma allows us to shorten the task of checking whether a graph homomorphism is a $k$-covering map or not.

\begin{lemma}[{\cite[Lemma~6.2]{MatsushitaJMSUT}}] \label{lemma checking criterion}
Let $p \colon G \to H$ be a graph homomorphism. Then the following are equivalent:
\begin{enumerate}[(1)]
\item The graph homomorphism $p$ is a $k$-covering map.

\item For every $x \in V(G)$, the map $p|_{N_G(x)} \colon N_G(x) \to N_H(x)$ is surjective and $p|_{N^k_G(x)} \colon N^k_G(x) \to N^k_H(p(x))$ is injective.
\end{enumerate}
\end{lemma}

\begin{lemma} \label{lemma short walk}
Let $\gamma \colon P_{2i} \to G$ be a closed walk of $(G,v)$ such that $i \le k$. Then $\gamma \simeq_ k *$. Here, $*$ denotes the trivial walk $P_0 \to G$ sending $0$ to $v$.
\end{lemma}
\begin{proof}
Define $\gamma' \colon P_{2i} \to G$ by
\[ \gamma'(x) = \begin{cases}
\gamma(x) & (x \le i) \\
\gamma(2i - x) & (x \ge i).
\end{cases}\]
Then $\gamma \simeq_k \gamma'$ by condition (B)$_k$ and it is clear that $\gamma' \simeq_k *$ by condition (A).
\end{proof}

The following proposition will be used in the subsequent section. Indeed, this proposition was essentially shown in \cite{MatsushitaJMSUT}, but here we provide its proof for the reader's convenience.

\begin{proposition} \label{proposition lifting}
Let $p \colon G \to H$ be a $k$-covering map, and $v_0$ a vertex of $G$.
\begin{enumerate}[(1)]
\item Let $\gamma$ be a walk in $H$ such that $\gamma(0) = p(v_0)$. Then there is a unique walk $\tgamma$ of $G$ such that $\tgamma(0) = v_0$ and $p \circ \tgamma = \gamma$. We call such $\tgamma$ the \emph{lift} of $\gamma$ starting at $v_0$.

\item Let $w_0 \in V(H)$ and let $\gamma, \gamma' \in \Omega(p(v_0), w_0)$. Let $\tgamma$ and $\tgamma'$ be the lifts of $\gamma$ and $\gamma'$ starting at $v_0$, respectively. If $\gamma \simeq_k \gamma'$, then the terminal points of $\tgamma$ and $\tgamma'$ coincide, and $\tgamma \simeq_k \tgamma'$.
\end{enumerate}
\end{proposition}
\begin{proof}
The statement (1) is easily checked and is actually a well-known fact of covering space theory (see \cite[Section~1.1]{Hatcher}, for example). Thus, we omit the detailed proof. We now show (2). It suffices to show the case that $\gamma$ and $\gamma'$ satisfy (A) or (B)$_k$ in Subsection~\ref{subsection k-fundamental group}. The case that $\gamma$ and $\gamma'$ satisfy (A) is easily checked and follows from the theory of covering maps. Consider the case that $\gamma$ and $\gamma'$ satisfy (B)$_k$. It suffices to see that in this case $\gamma(i) = \gamma'(i)$ implies $\tgamma(i) = \tgamma'(i)$. Suppose that there is $i$ such that $\gamma(i) = \gamma'(i)$ but $\tgamma(i) \ne \tgamma'(i)$. Take $i_0$ to be the  minimum such $i$. If $i_0 \le k$, then $\tgamma(i_0), \tgamma'(i_0) \in N^{i_0}(v_0)$ and $p(\tgamma(i_0)) = p(\tgamma'(i_0))$. Since $p$ is a $k$-covering, this is a contradiction. Hence we have $i_0 > k$. Since $\gamma$ and $\gamma'$ satisfy (B)$_k$, there is $j < i_0$ such that $\gamma(j) = \gamma'(j)$ and $i_0 - j < k$. By the definition of $i_0$, we have $\tgamma(j) = \tgamma'(j)$. Hence we have $\tgamma(i_0), \tgamma'(i_0) \in N^{i_0 - j}(\tgamma(j))$. Since $p$ is a $k$-covering, this means $\tgamma(i_0) = \tgamma'(i_0)$, and this is again a contradiction. This completes the proof.
\end{proof}

\section{Proof of the main theorem} \label{section Proof}

The goal of this section is to prove Theorem~\ref{theorem main}.

\subsection{Borsuk graphs} \label{subsection Borsuk graph}

While there are several equivalent formulations of circular colorings and circular chromatic numbers (see \cite{Zhu}), in this paper we use the following formulation, using the Borsuk graph.

In this paper, we regard $S^1$ as $\RR / \ZZ$, and consider the metric $d$ on $S^1$ induced by the Euclidean metric: For $\alpha, \beta \in S^1 = \RR / \ZZ$, set
\[ d(\alpha, \beta) = \min \{ |b-a| \mid a \in \alpha, b \in \beta \}.\]
For a positive number $a$, Define the \emph{Borsuk graph $\Bor(S^1; a)$} as follows:
\[ V(\Bor(S^1 ; a)) = S^1,\quad E(\Bor(S^1 ; a)) = \{ (x,y) \in S^1 \mid d(x,y) \ge a \}.\]

\begin{definition} \label{definition circular chromatic number}
Let $G$ be a graph. Define the circular chromatic number $\chi_c(G)$ as follows:
\[ \chi_c(G) = \inf \{ r \ge 1 \mid \textrm{there is a graph homomorphism from $G$ to $\Bor(S^1 ; r^{-1})$}\}\]
\end{definition}

The main task of the proof of Theorem~\ref{theorem main} is to show the following theorem, which determines the $k$-fundamental group of $\Bor(S^1; r^{-1})$ for every $k \ge 1$.

\begin{theorem} \label{theorem pi_1}
Let $r$ be a real number greater than $2$, and set $q = \left\lceil \frac{2}{r - 2} \right\rceil $. Then, for an integer $k \ge 2$, we have
\[ \pi_1^k (\Bor(S^1 ; r^{-1})) \cong \begin{cases}
\ZZ & (k \le q) \\
\ZZ / 2 & (k > q).
\end{cases}\]
\end{theorem}

The proof of this theorem is postponed to the next subsection. In this subsection, we complete the proof of Theorem~\ref{theorem main}, assuming Theorem~\ref{theorem pi_1}.

\begin{theorem} \label{theorem circular 1}
Let $G$ be a connected graph and $k$ an integer at least $2$. Assume that there is a closed walk $\gamma$ in $G$ with odd length such that the element in $\pi_1^k(G) / [\pi_1^k(G), \pi_1^k(G)]$ represented by $[\gamma]_k \in \pi_1^k(G)$ is a torsion element. Then we have
\[ \chi_c(G) \ge 2 + \frac{2}{k-1}.\]
\end{theorem}
\begin{proof}[Proof of Theorem~\ref{theorem circular 1} under Theorem~\ref{theorem pi_1}]
Suppose that $\chi_c(G) < 2 + \frac{2}{k-1}$. Then there is $r < 2 + \frac{2}{k-1}$ such that there is a graph homomorphism $f$ from $G$ to $\Bor(S^1 ; r^{-1})$. Consider the group homomorphism $f_* \colon \pi_1^k(G) \to \pi_1^k(\Bor(S^1; r^{-1}))$. By Theorem~\ref{theorem pi_1}, $\pi_1^k(\Bor(S^1; r^{-1})) \cong \ZZ$. Since $\ZZ$ is an abelian group, $f_*$ factors through $\pi_1^k(G) / [\pi_1^k(G), \pi_1^k(G)]$. Since $[\gamma]_k$ is a torsion element in $\pi_1^k(G) / [\pi_1^k(G), \pi_1^k(G)]$, $f_*([\gamma]_k)$ is trivial. On the other hand, since $f \circ \gamma$ is a closed walk with odd length, $[f \circ \gamma]_k = f_* ([\gamma]_k)$ is not trivial (see Remark~\ref{remark parity}). This is a contradiction.
\end{proof}

Now we prove the main theorem, assuming Theorem~\ref{theorem pi_1}.

\begin{proof}[Proof of Theorem~\ref{theorem main} under Theorem~\ref{theorem pi_1}]
Set $G = X^1$. If there exist $2$-cells in $X$ whose attaching maps coincide, let $X'$ denote the $2$-dimensional complex obtained by removing all but one of them from $X^2$. By the definition of cellular homology, we have an isomorphism $H_1(X; \mathbb{Z}) \cong H_1(X'; \mathbb{Z})$, and $X'$ is a subcomplex of $X_k(X^1)$.
Let $\iota$ be the inclusion $X' \hookrightarrow X_k(X^1)$.
Since $X_k(X^1)^1 = X^1 = G$, the map $\iota_* \colon \pi_1(X') \to \pi_1(X_k(G))$ induced by $\iota$ is surjective (Proposition~\ref{proposition 2-connected}).
Recall that $\pi_1(X_k(G)) \cong \pi_1^k(G)$. Hence we obtain a natural surjective group homomorphism
\[ \pi_1(X') \to \pi_1^k(G).\]
Considering the composition of the sequence (see also Theorem~\ref{theorem abelianization})
\[
H_1(X; \mathbb{Z}) \cong H_1(X'; \ZZ) \cong \pi_1(X') / [\pi_1(X'), \pi_1(X')] \to \pi_1^k(G) / [\pi_1^k(G), \pi_1^k(G)],
\]
the homology class $[\gamma] \in H_1(X; \mathbb{Z})$ is sent to \([\gamma]_k \in \pi_1^k(G) / [\pi_1^k(G), \pi_1^k(G)]\).
Hence, \([\gamma]_k\) is a torsion element in \(\pi_1^k(G) / [\pi_1^k(G), \pi_1^k(G)]\), and the proof is completed by Theorem~\ref{theorem circular 1}.
\end{proof}

Hence, to complete the proof of Theorem~\ref{theorem main}, it suffices to show Theorem~\ref{theorem pi_1}.

Before concluding this subsection, we discuss the relationship between Theorem \ref{theorem pi_1} and existing results. Recently, Krebs and Sankar \cite[Theorem~1.4]{KrebsSankar} showed that if $G$ is a connected non-bipartite graph and the first integral homology group $H_1(\N(G) ; \ZZ)$ of $\N(G)$ is a torsion group, then $\chi(G) \ge 4$. This result can be refined by Theorem~\ref{theorem pi_1} in terms of the circular chromatic number as follows:

\begin{corollary} \label{corollary KS}
Let $G$ be a connected non-bipartite graph. If $H_1(\N(G); \ZZ)$ is a torsion group, then $\chi_c(G) \ge 4$.
\end{corollary}
\begin{proof}
By Theorems~\ref{theorem abelianization} and \ref{theorem neighborhood complex}, the group $\pi_1^2(G) / [\pi_1^2(G)_{ev}, \pi_1^2(G)_{ev}]$ includes
\[ \pi^2_1(G)_{ev} / [\pi^2_1(G)_{ev}, \pi^2_1(G)_{ev}] \cong H_1(\N(G); \ZZ)\]
as a subgroup of finite index, and hence is a torsion group. Since there is a surjective group homomorphism $\pi^2_1(G) / [\pi^2_1(G)_{ev}, \pi^2_1(G)_{ev}] \to \pi^2_1(G) / [\pi^2_1(G), \pi^2_1(G)]$, the group $\pi^2_1(G) / [\pi^2_1(G), \pi^2_1(G)]$ is also a torsion group. Thus Theorem~\ref{theorem circular 1} completes the proof.
\end{proof}

\begin{remark}
Here, we discuss the relationship between Theorem~\ref{theorem pi_1} and \v{C}ech complexes $\bC_{\le}(S^1; s)$ of a circle \cite{AA}. For $s \ge 0$, the \emph{\v{C}ech complex $\bC_{\le}(X; s)$ of a metric space $(X, d)$} is the simplicial complex whose vertex set is $X$, and where a finite subset of $X$ forms a simplex of $\bC(X; s)$ if it is contained in some closed ball of radius $s$. The \v{C}ech complex is a fundamental concept in applied topology and has been extensively studied (see, for example, \cite{Carlsson, EH}). It is straightforward to verify that $\N^j(\Bor(S^1; a))$ coincides with the \v{C}ech complex $\bC_{\le} \left(S^1; j(\frac{1}{2} - a) \right)$ of the circle.

The homotopy type of the \v{C}ech complex $\bC_{\le}(S^1; s)$ was determined by Adamaszek and Adams. According to their result \cite[Theorem~1.1]{AA}, $\bC_{\le}(S^1; s)$ is homotopy equivalent to $S^1$ for $0 < s < \frac{1}{4}$, and is simply connected for $s \ge \frac{1}{4}$. Hence, Theorem~\ref{theorem neighborhood complex} implies
\[
\pi_1^{2j}(\Bor(S^1 ; r^{-1}))_{ev} \cong \begin{cases}
\ZZ & \text{if } 2j \le q \\
0 & \text{if } 2j > q,
\end{cases}
\]
where $q = \left\lceil \frac{2}{r - 2} \right\rceil $. One can deduce Theorem~\ref{theorem pi_1} from this computation when $k$ is even. In the next subsection, we give a proof valid for both even and odd $k$.
\end{remark}

\subsection{Proof of Theorem~\ref{theorem pi_1}}

The goal of this section is to prove Theorem~\ref{theorem pi_1}, which determines the $k$-fundamental groups of $\Bor(S^1; r^{-1})$.

Throughout this section, symbols $r$, $a$ and $q$ frequently appear to indicate real numbers. The relations between these numbers are as follows: $r$ is a real number greater than $2$, $a = r^{-1}$ and $q = \left\lceil \frac{2}{r-2} \right\rceil$. Then, note that $0 < a < 1/2$.
To determine $\pi_1^k(\Bor(S^1; a))$, we construct the ``universal covering" of $\tBor(S^1; a)$ as follows:

Suppose that $0 < a < 1/2$. Define the graph $\tBor(S^1; a)$ by
\[ V(\tBor(S^1; a)) = \RR \times \{ \pm 1\},\]
\[ E(\tBor(S^1; a)) = \left\{ \left((x,\varepsilon), (x', \varepsilon') \right) \mid |x - x'| \le \frac{1}{2} - a,\ \varepsilon' = - \varepsilon \right\}.\]
In the following, we often write $\tB_a$ and $B_a$ instead of $\tBor(S^1 ; a)$ and $\Bor(S^1; a)$, respectively.

\begin{lemma} \label{lemma bipartite}
The graph $\tB_a = \tBor(S^1; a)$ is bipartite for $0 < a < 1/2$.
\end{lemma}
\begin{proof}
Define $f \colon \tB_a \to K_2$ by $f(x,1) = 1$ and $f(x, - 1) = 2$ for every $x \in \RR$.
\end{proof}

Define the map $p \colon V(\tBor(S^1 ; a)) \to V(\Bor(S^1 ; a))$ by
\[ p(x, 1) = x \mod 1 \quad \textrm{and} \quad p(x, -1) = x + \frac{1}{2} \mod 1.\]

\begin{lemma} \label{lemma surjective}
If $0 < a < 1/2$, then for every $x \in V(\tBor(S^1; a))$, we have $p(N_{\tB_a}(x)) = N_{B_a}(p(x))$. In particular, $p$ is a graph homomorphism.
\end{lemma}
\begin{proof}
Let $(x,1) \in V(\tB_a)$. Then
\[ N_{\tB_a}(x,1) = \Big\{ (y, -1) \mid - \Big( \frac{1}{2} - a \Big) \le y - x  \le  \frac{1}{2} - a \Big\}.\]
Hence we have
\begin{eqnarray*}
p(N_{\tB_a}(x,1)) &=& \Big\{ \ y + \frac{1}{2} \hspace{-1mm}\mod 1 \mid - \Big( \frac{1}{2} - a \Big) \le y - x  \le \frac{1}{2} - a \Big\} \\
&=& \big\{ \ y \hspace{-1mm}\mod 1 \mid a \le y - x \le 1 - a \big\} \\
&=& N_{B_a}\big( x \hspace{-1mm}\mod 1 \big)\\
&=& N_{B_a}(p(x,1)).
\end{eqnarray*}
The case of $(x,-1)$ is similar and straightforward, and we omit the details.
\end{proof}

\begin{lemma} \label{lemma simply-connected}
Suppose $0 < a < 2^{-1}$. Then $\pi_1^k(\tBor(S^1; a))$ is trivial for every $k \ge 2$.
\end{lemma}
\begin{proof}
Since there is a natural surjection $\pi_1^2(\tBor(S^1; a)) \to \pi_1^k(\tBor(S^1;a))$, it suffices to show that $\pi_1^2(\tBor(S^1; a))$ is trivial. Set $v = (0,1) \in V(\tBor(S^1; a))$ as a basepoint of $\tB_a$.

Let $\gamma$ be a closed walk whose basepoint is $v$, and suppose that the length $n$ of $\gamma$ is greater than $0$. Define $\pr_1 \colon V(\tBor(S^1; a)) \to \RR$ by $\pr_1(x, \varepsilon) = x$. Let $M(\gamma)$ be the maximum of $\{ |\pr_1(\gamma(0))|, \cdots, |\pr_1(\gamma({\rm length}(\gamma)))|\}$. It is clear that if $M(\gamma) = 0$, then $[\gamma]_2$ is trivial.

Set $b = \frac{1}{2} - a$. Let $i \in \{ 0, \cdots, n\}$ such that $|\pr_1(\gamma(i))| = M(\gamma)$. Define $\gamma'$ as follows:
\begin{itemize}
\item Suppose $\gamma(i) \ge b$. Then define $\gamma'$ by
\[ \left( \gamma(0), \cdots, \gamma(i-1), \gamma(i) - b, \gamma(i+1), \cdots, \gamma(n) \right).\]

\item Suppose $\gamma(i) \le -b$. Then define $\gamma'$ by
\[ \left( \gamma(0), \cdots, \gamma(i-1), \gamma(i) + b, \gamma(i+1), \cdots, \gamma(n)\right).\]

\item $-b \le \gamma(i) \le b$. Then define $\gamma'$ by
\[ \left(\gamma(0), \cdots, \gamma(i-1), 0, \gamma(i+1), \cdots, \gamma(n)\right).\]
\end{itemize}
By the definition, we have $\gamma \simeq_2 \gamma'$. Repeating the modifications above, we have a closed walk $\gamma''$ such that $M(\gamma'') = 0$. This completes the proof.
\end{proof}

Define a graph homomorphism $f \colon \tBor(S^1 ; a) \to \tBor(S^1 ; a)$ by
\[ f(x, \varepsilon) = \left(x + \frac{1}{2}, -\varepsilon\right).\]
Then $f$ is an isomorphism of graphs, and
\[ f^n(x, \varepsilon) = \left(x + \frac{n}{2}, (-1)^n \varepsilon\right).\]

\begin{theorem} \label{theorem pi_1a}
Suppose that $0 < a < \frac{1}{2}$, and set $b = \frac{1}{2} - a$. Let \( q' \) be the smallest integer satisfying \( 2q'b \ge 1 \), that is, \( q' = \left\lceil \frac{1}{2b} \right\rceil \).
Then the graph homomorphism $p \colon \tBor(S^1 ; a) \to \Bor(S^1 ; a)$ is a \((q' - 1)\)-covering map, but not a \(q'\)-covering map.
\end{theorem}
\begin{proof}
We first show that $p$ is not a $q'$-covering. By the definition of $q'$, there is a walk $\tgamma \colon P_{2q'} \to \tBor(S^1; a)$ such that $\tgamma(0) = (0,1)$ and $\tgamma(2q) = (1,1)$. Then we have $(0,1), (1,1) \in N^{q'}(\tgamma(k))$ but $p\left( (0,1) \right) = p\left( (1,1)\right)$. This means that $p$ is not a $q'$-covering.

Next we show that $p$ is a $(q'-1)$-covering. Let $v = (x, \varepsilon) \in V(\tBor(S^1; a))$. By Lemmas~\ref{lemma checking criterion} and \ref{lemma surjective}, it suffices to show that $p |_{N^{q'-1}(v)} \colon N^{q'-1}(v) \to N^{q'-1}(p(v))$ is injective. Suppose that there are distinct vertices $v_0, v_1 \in N^{q'-1}(v)$ such that $p(v_0) = p(v_1)$. Then there are distinct $x_0, x_1 \in \RR$ such that $v_0 = (x_0, (-1)^{q'-1} \varepsilon)$ and $v_1 = (x_1, (-1)^{q'-1} \varepsilon)$, and $|x_1 - x_0| \ge 1$. Since $v_0, v_1 \in N^{q'-1}(v)$, there is a walk $\tgamma \colon P_{2q'-2} \to \tBor(S^1 ; a)$ such that $\tgamma(0) = v_0$ and $\tgamma(2q'-2) = v_1$. However, since $(2q' -2)b < 1$, this means that $|x_1 - x_0| < 1$. This is a contradiction.
\end{proof}

Since $2 < r$, $a = r^{-1}$, and $b = \frac{1}{2} - a$, we have
\[ q' = \left\lceil \frac{1}{2b} \right\rceil = \left\lceil \frac{2}{r - 2} \right\rceil + 1 = q + 1.\]
Therefore, Theorem~\ref{theorem pi_1a} implies the following:

\begin{corollary} \label{corollary covering}
Let $r$ be a real number greater than $2$, and set \( q = \left\lceil \frac{2}{r - 2} \right\rceil \). Then the graph homomorphism $p \colon \tBor(S^1 ; r^{-1}) \to \Bor(S^1 ; r^{-1})$ is a \(q\)-covering, but not a $(q + 1)$-covering.
\end{corollary}

Now we are ready to prove Theorem~\ref{theorem pi_1}:

\begin{proof}[Proof of Theorem~\ref{theorem pi_1}]
Let $k \ge 2$ be an integer. We first show that $\pi_1^k(\Bor(S^1; a)) \cong \ZZ / 2$ for $k > q$. Let $v = 0 \in V(\Bor(S^1; a))$ and $\tv = (0, 1) \in \tBor(S^1; a)$. Since $0 < a < \frac{1}{2}$, the graph $\Bor(S^1; a)$ is non-bipartite and hence $\pi_1^k(\Bor(S^1, a); v)$ has an odd element. Hence it suffices to show that $\pi_1^k(\Bor(S^1, a); v)_{ev}$ is trivial.

By the assumption on $q$, there is a walk $\delta$ of length $2q$ in $\tBor(S^1; a)$ joining $(0,1)$ to $(1,1)$. Then, $p \circ \delta$ is a closed walk of length $2k$, and hence $[\delta]_i$ is trivial (Lemma~\ref{lemma short walk}).

Let $[\gamma] \in \pi_1^i(\Bor(S^1; a), v)_{ev}$. Since $p \colon \tB_a \to B_a$ is a $1$-covering map, there is a walk $\tgamma$ starting at $\tv$ such that $p \circ \tgamma = \gamma$ (Proposition~\ref{proposition lifting}). Since the length of $\gamma$ is even, there is an integer $m$ such that the terminal point of $\tgamma$ is $(m, 1)$.
If $m \ge 0$, then by Lemma~\ref{lemma simply-connected}, we have $\tgamma \simeq_k \delta \cdot (f \circ \delta) \cdots (f^{m-1} \circ \delta)$. Hence,
\[ \gamma = p \circ \tgamma \simeq_k p \circ \left( \delta \cdot (f \circ \delta) \cdots (f^{m-1} \circ \delta) \right) = (p \circ \delta) \cdot (p \circ \delta) \cdots (p \circ \delta) \simeq_k *.\]
This completes the proof that $\pi_1^k(\Bor(S^1; a), v)_{ev}$ is trivial, and hence $\pi_1^k(\Bor(S^1; a)) \cong \ZZ / 2$ for $k > q$.

Next suppose $2 \le k \le q$. We define the bijection $\psi \colon \pi_1^i(\Bor(S^1 ; r^{-1})) \to \ZZ$ as follows. By Corollary~\ref{corollary covering}, $p \colon \tBor(S^1; r^{-1}) \to \Bor(S^1; r^{-1})$ is an $k$-covering.
Let $\gamma$ be a closed walk in $\Bor(S^1;a)$ starting at $v = 0$.
Let $\tgamma$ be the lift of $\gamma$ with respect to $p \colon \tBor(S^1; a) \to \Bor(S^1; a)$ starting at $\tv = (0,1) \in \tBor(S^1; a)$ (Proposition~\ref{proposition lifting}). Then the terminal point $\tw$ of $\tgamma$ is contained in $p^{-1}(0)$, and hence there is a unique integer $\psi(\gamma)$ such that $\tw = f^{\psi(\gamma)}(0,1)$. By Proposition~\ref{proposition lifting}, this integer $\psi(\gamma)$ depends only on the $k$-homotopy class $[\gamma]_k$ of $\gamma$. This induces a map $\psi \colon \pi_1^k(\Bor(S^1; v)) \to \ZZ$.

This $\psi$ is a group homomorphism. Indeed, let $\gamma$ and $\gamma'$ be closed walks of $(\Bor(S^1 ; a), v)$. Then $\tgamma \cdot (f^{\psi(\gamma)} \circ \tgamma')$ is the lift of $\gamma \cdot \gamma'$ starting at $\tv$, and its terminal point is $f^{\psi(\gamma) + \psi(\gamma')}(v)$. Hence we have $\psi ([\gamma]_i \cdot [\gamma']_k) = \psi([\gamma]_k) + \psi([\gamma']_k)$.

Next we show that $\psi$ is surjective. To see this, for $n \in \ZZ$, let $\tgamma$ be a walk in $\tBor(S^1 ; r^{-1})$ joining $\tv$ to $f^n(\tv)$. Then $n = \psi ([p \circ \tgamma]_k)$, which shows that $\psi$ is surjective.

Finally, we show that $\psi$ is injective. Let $[\gamma]_k \in \Ker(\psi)$. Then $\tgamma$ is a closed walk of $(\tBor(S^1; r^{-1}), 0)$. Since $k \ge 2$, we have $\pi_1^k(\tBor(S^1; r^{-1}))$ is trivial and hence we have $[\gamma]_k = p_*([\tgamma]_i)$ is also trivial. This completes the proof that $\pi_1^k(\Bor(S^1; r^{-1})) \cong \ZZ$ for $k \le q$.
\end{proof}

By combining the discussion in Subsection~\ref{subsection Borsuk graph}, we complete the proof of Theorem~\ref{theorem main}. We end this paper by discussing the \(k\)-fundamental groups of circular complete graphs $K_{n/m}$ (see \cite{PW, Zhu, Zhu2}). Let \(n\) and \(m\) be positive integers. Define the metric \(d\) on \(\ZZ / n\ZZ\) by
\[ d(x,y) = \min \{ i \in \ZZ_{\ge 0} \mid \text{$x = y + i$ or $x = y - i$ in $\ZZ / n \ZZ$} \}.\]
The \emph{circular complete graph \(K_{n/m}\)} is defined by
\[ V(K_{n/m}) = \ZZ/ n, \quad E(K_{n/m}) = \{ \{ x,y\} \mid x,y \in \ZZ / n \ZZ,\ d(x,y) \ge m \}.\]
Although we defined the circular chromatic number $\chi_c(G)$ using the Borsuk graph, it can also be defined using circular complete graphs in a similar way (see \cite{Zhu}).

\begin{proposition}
Let $n$ and $m$ be positive integers. Set $r = \frac{n}{m}$ and suppose that $2 < r < 4$. Set $q = \left\lceil \frac{2}{r - 2} \right\rceil $. Then, for any integer $k \ge 2$, we have
\[ \pi_1^k (K_{n/m}) \cong \begin{cases}
\ZZ & (2 \le k \le q) \\
\ZZ / 2 & (k > q).
\end{cases}\]
\end{proposition}
\begin{proof}
By Theorem~\ref{theorem pi_1}, it suffices to show that $\pi_1^k(K_{n/m}) \cong \pi_1^k(\Bor(S^1; r^{-1}))$. Define a graph homomorphism \( f \colon K_{n/m} \to \Bor(S^1; r^{-1}) \) by
\[
f \big(x \hspace{-2mm} \mod n \big) = \frac{x}{n} \hspace{-1mm}\mod 1
\]
for $x \in \ZZ$, and define a graph homomorphism \( g \colon \Bor(S^1 ; r^{-1}) \to K_{n/m} \) by
\[
g\big( x \hspace{-2mm}\mod 1 \big) = \lceil nx \rceil \hspace{-1mm}\mod n
\]
for $x \in \RR$. Then we have \( gf = \mathrm{id}_{K_{n/m}} \). Hence, \(\pi_1^k(K_{n/m})\) is a retract of \(\pi_1^k(\Bor(S^1 ; r^{-1}))\), and since \(K_{n/m}\) is non-bipartite, \(\pi_1^k(K_{n/m})\) is not a trivial group. From Theorem~\ref{theorem pi_1}, we know that \(\pi_1^k(\Bor(S^1 ; r^{-1}))\) is isomorphic to $\ZZ$ or $\ZZ / 2$, so in these cases, \(f_* \colon \pi_1^k(K_{n/m}) \to \pi_1^k(\Bor(S^1 ; r^{-1}))\) is an isomorphism.
\end{proof}

\section{Application to the graph coloring problem of Cayley graphs}

In this section, we discuss the relationship between our results and the recent work of Krebs and Sankar \cite{KrebsSankar}. Let $\Gamma$ be a group. A subset $S$ of $\Gamma$ is said to be \emph{symmetric} if $x \in S$ implies $x^{-1} \in S$. The Cayley graph $\Cay(\Gamma, S)$ with symmetric set $S$ is defined to be the graph whose vertex set is $\Gamma$ and whose edge set is $\{ \{x,y\} \mid x^{-1} y \in S \}$.

The graph coloring problem of Cayley graphs has been extensively studied (see \cite{Alon, Czerwinski} for example). Recently, Krebs and Sankar \cite{KS} showed the following striking theorem:

\begin{theorem}[{\cite[Theorem~1.2]{KrebsSankar}}] \label{theorem KS}
Let $\Gamma$ be an abelian group. Then, there exists a symmetric subset $S$ such that $\chi(\Cay(\Gamma, S)) = 3$ if and only if $\exp(\Gamma) \not\in \{ 1,2,4\}$.
\end{theorem}

Here, recall that the exponent $\exp(\Gamma)$ of a group $\Gamma$ is the smallest positive integer $n$ such that $g^n = e_\Gamma$ for every $g \in \Gamma$.

Let $\Gamma_0$ be the subgroup of $\Gamma$ generated by $S$. Then $\Cay(\Gamma, S)$ is a disjoint union of copies of $\Cay(\Gamma_0, S)$. Krebs and Sankar in fact showed that $\pi_1^2(\Cay(\Gamma_0, S))$ is a torsion group (see Proof of Theorem~1.2 of \cite{KrebsSankar}). Thus, by Theorem~\ref{theorem circular 1}, we have the following refinement of one direction of Theorem~\ref{theorem KS}.


\begin{theorem} \label{theorem refinement}
Let $\Gamma$ be an abelian group. If $\exp(\Gamma) \in \{ 1,2,4\}$ and $\Cay(\Gamma, S)$ is non-bipartite, then $\chi_c(\Cay(\Gamma, S)) \ge 4$.
\end{theorem}

\section*{Acknowledgement}
The first author and the second author were supported in part by JSPS KAKENHI Grant Numbers JP23K13006 and JP23K12975, respectively. The authors are grateful to Yuta Nozaki and Kenta Ozeki for discussions regarding their recent work \cite{Kaiseretal}. The authors are grateful to the anonymous referees for their insightful and helpful comments. In particular, one of the referees kindly brought the recent work \cite{KrebsSankar} of Krebs and Sankar to our attention.

\section*{Data availability}
No data was used for the research described in the article.

\section*{Declaration of generative AI and AI-assisted technologies in the manuscript preparation process}

During the preparation of this work the authors used Gemini in order to improve the readability of the manuscript. After using these tools, the authors reviewed and edited the content as needed and takes full responsibility for the content of the published article.

\bibliographystyle{abbrv}

\end{document}